\documentclass[oneside,english]{amsart}

\usepackage[T1]{fontenc}
\usepackage[latin9]{inputenc}
\usepackage{refstyle}
\usepackage{mathrsfs}
\usepackage{amsthm}
\usepackage{amssymb}
\usepackage{stmaryrd}
\usepackage{wasysym}
\usepackage[all]{xy}
\usepackage{amsmath}

\makeatletter

\AtBeginDocument{\providecommand\remref[1]{\ref{rem:#1}}}
\AtBeginDocument{\providecommand\thmref[1]{\ref{thm:#1}}}
\AtBeginDocument{\providecommand\lemref[1]{\ref{lem:#1}}}
\AtBeginDocument{\providecommand\secref[1]{\ref{sec:#1}}}
\RS@ifundefined{subref}
  {\def\RSsubtxt{section~}\newref{sub}{name = \RSsubtxt}}
  {}
\RS@ifundefined{thmref}
  {\def\RSthmtxt{theorem~}\newref{thm}{name = \RSthmtxt}}
  {}
\RS@ifundefined{lemref}
  {\def\RSlemtxt{lemma~}\newref{lem}{name = \RSlemtxt}}
  {}

\theoremstyle{plain}
\newtheorem{thm}{\protect\theoremname}[section]
  \theoremstyle{remark}
  \newtheorem{rem}[thm]{\protect\remarkname}
  \theoremstyle{definition}
  \newtheorem{defn}[thm]{\protect\definitionname}
  \theoremstyle{plain}
  \newtheorem{lem}[thm]{\protect\lemmaname}

\usepackage{amsmath}
\usepackage{mathtools}

\newref{prop}{name=Proposition~}
\newref{lem}{name=Lemma~}
\newref{thm}{name=Theorem~}
\newref{rem}{name=Remark~}
\newref{cor}{name=Corollary~}
\newref{defn}{name=Definition~}

\def\tobar{\mathrel{\mkern3mu  \vcenter{\hbox{$\scriptscriptstyle+$}}%
                    \mkern-12mu{\to}}}

\makeatother

\usepackage{babel}
  \providecommand{\definitionname}{Definition}
  \providecommand{\lemmaname}{Lemma}
  \providecommand{\remarkname}{Remark}
\providecommand{\theoremname}{Theorem}

\begin{document}

\title{An ordered framework for partial multivalued functors}

\author{{Alveen Chand} and {Ittay Weiss}}

\maketitle

\begin{abstract}
The category ${\bf Rel}$ of sets and relations intimately ties the
notions of function, partial multivalued function, and direct image
under a function through the description of ${\bf Rel}$ as the Kleisli
category of the covariant power set functor on ${\bf Set}$. We present
a suitable framework to obtain a similar relationship between the
concepts of functor, partial multivalued functor, and the direct image
under a functor. 
\end{abstract}

\section{Introduction\label{sec:Introduction}}

Partial multivalued functions arise naturally in the presence of uncertainty
or partial information. In computer science, we mention \cite{taxonomy}
in the context of complexity theory and \cite{relDatabse} in the
context of multivalued relational databases, where partial multivalued
functions appear prominently. In mathematics, multivalued functions
are common-place in complex analysis and used in homotopy theory.
Mathematically, a partial multivalued function $S\tobar T$ can be
modeled as a relation from $S$ to $T$ or as a function $S\to\mathcal{P}(T)$
to the power set of $T$. Often, the sets $S$ and $T$ are endowed
with extra structure, for instance an ordering, and the partial multivalued
functions are to preserve that information in some sense. More generally,
the domain $S$ and the codomain $T$ may be categories, and instead
of a partial multivalued function one is interested in the notion
of a partial multivalued functor. We now take two straightforward
ad-hoc approaches to model the notion of a partial multivalued functor. 

Extending the notion of a partial multivalued function to categories,
we define a \emph{partial multivalued functor} $F\colon\mathscr{C}\tobar\mathscr{D}$
between two small categories to consist of a partial multivalued function
$F\colon{\rm ob}\mathscr{C}\tobar{\rm ob}\mathscr{D}$ and a partial
multivalued function $F\colon\mathscr{C}(C,C')\tobar\bigcup_{D\in FC,D'\in FC'}\mathscr{D}(D,D')$,
for all $C,C'\in{\rm ob}\mathscr{C}$, such that the identity morphisms
and composition are respected, in the sense that 
\begin{itemize}
\item for all $C\in{\rm ob}\mathscr{C}$, if $d\in F({\rm id}_{C})$, then
$d={\rm id}_{D}$ for some $D\in FC$; and
\item for all $c''\in{\rm mor}\mathscr{C}$ and $d''\in{\rm mor}\mathscr{D}$
with $d''\in F(c'')$, if $c''=c'\circ c$ for some $c,c'\in{\rm mor}\mathscr{C}$,
then there exist $d\in Fc$ and $d'\in Fc'$ with $d''=d'\circ d$. \end{itemize}
\begin{rem}
\label{rem:NotepmvFunctor}Note that the naively more immediate condition
for preservation of composition, namely that if $d\in F(c)$ and $d'\in F(c')$,
then $d'\circ d\in F(c'\circ c)$, is not a natural condition to impose.
Indeed, the composability of $c'$ with $c$ need not imply that of
$d'$ with $d$, and it seems rather contrived to demand that if $d'\circ d$
happens to exist, then $d'\circ d\in F(c'\circ c$). 
\end{rem}
With the obvious notions of identities and composition, all small
categories and partial multivalued functors form the category ${\bf Cat_{pmv}}$.
Obviously, a functor is (by slight abuse of terminology) also a partial
multivalued functor, and thus ${\bf Cat}$ is a subcategory ${\bf Cat_{pmv}}$.

The second approach is obtained by extending the notion of relation
between sets to categories. A \emph{relation} $R\colon\mathscr{C}\tobar\mathscr{D}$
between two small categories consists of a relation $R\colon{\rm ob}\mathscr{C}\tobar{\rm ob}\mathscr{D}$
and a relation $R\colon\mathscr{C}(C,C')\tobar\bigcup_{DRC,D'RC'}\mathscr{D}(D,D')$,
for all $C,C'\in{\rm ob}\mathscr{C}$, which are compatible with the
identities and the compositions, in the sense that
\begin{itemize}
\item for all $C\in{\rm ob}\mathscr{C}$, if ${\rm id}_{C}Rd$, then $d={\rm id}_{D}$
for some $D\in{\rm ob}\mathscr{D}$ with $CRD$; and
\item for all $c''\in{\rm mor}\mathscr{C}$ and $d''\in{\rm mor}\mathscr{D}$
with $c''Rd''$, if $c''=c'\circ c$ for some $c,c'\in{\rm mor}\mathscr{C}$,
then $d''=d'\circ d$ for some $d,d'\in{\rm mor}\mathscr{D}$ with
$cRd$, $c'Rd'$. \end{itemize}
\begin{rem}
Analogously to \remref{NotepmvFunctor}, the more naive demand that
$R$ be a congruence for the composition, namely that if $cRd$ and
$c'Rd'$, then $(c'\circ c)R(d'\circ d)$ suffers from the same ill-behaviour
that the naive preservation of composition exhibits. We further note
that the temptation of defining a relation from $\mathscr{C}$ to
$\mathscr{D}$ to be a subcategory of $\mathscr{C}\times\mathscr{D}$
results in a notion that is hardly related to our notion of relation,
as can easily be seen by direct inspection. 
\end{rem}
With the evident notions of identities and composition, all small
categories and relations form the category ${\bf Cat_{rel}}$. 
\begin{thm}
The categories ${\bf Cat_{pmv}}$ and ${\bf Cat_{rel}}$ are isomorphic. \end{thm}
\begin{proof}
An isomorphism is given by the identity on objects, mapping a partial
multivalued functor $F\colon\mathscr{C}\tobar\mathscr{D}$ to the
relation $R$ where $CRD$ if $D\in FC$, and $cRd$ if $d\in Fc$,
respectively, for all objects and morphisms. The details are immediate. 
\end{proof}
The notions of function, direct image under a function, and partial
multivalued function (i.e., a relation) are related in the following
way. Consider the category ${\bf Set}$ of sets and functions and
recall the covariant power set functor $\mathcal{P}\colon{\bf Set}\to{\bf Set}$,
mapping a set to its power set and a function $f\colon S\to T$ to
the direct image function $f_{\shortrightarrow}\colon\mathcal{P}S\to\mathcal{P}T$.
This endofunctor is well-known to be part of a monad, known as the
Manes monad, whose Kleisli category is ${\bf Rel}$, the category
of sets and relations. In light of the above construction, it is natural
to expect the category ${\bf Cat_{pmv}}$, and thus also its
isomorphic copy ${\bf Cat_{rel}}$, to arise as the Kleisli category
of a monad on ${\bf Cat}$ given on objects by some sort of power
category construction, and on functors by a suitable direct image
construction. 

However, there is an immediate obstruction to such a result. It is
well known that the naive notion of image of a functor need not be
a subcategory of the codomain. A minimal example illustrating this
is the functor $K$ 
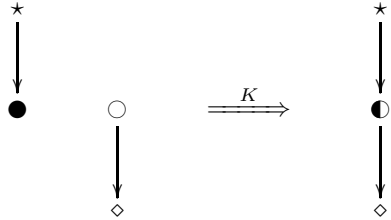
\[
\xymatrix{\star\ar[d] &  &  & \star\ar[d]\\
\CIRCLE & \Circle\ar[d] & \xRightarrow{\quad K\quad} & \LEFTcircle\ar[d]\\
 & \diamond &  & \diamond
}
\]
which is not expected to readily be reconciled with a covariant power
category functor on ${\bf Cat}$.

The aim of this work is to present a category-based formalism extending
the function/direct-image/partial-multivalued-function trio to a corresponding
menage a trois of the concepts of functor, direct image under a functor,
and partial multivalued functor. In more detail, denoting ${\bf Set}$
by ${\bf Cat_{0}}$ and ${\bf Rel}$ by ${\bf Rel_{0}}$, we construct
suitable categories ${\bf Cat_{1}}$ and ${\bf Rel_{1}}$, whose objects
are categories, and we obtain the diagram

\[
\xymatrix{{\bf Cat_{0}}\ar@(ul,dl)[]_{\mathcal{P}_{0}}\ar[d] & {\bf Cat_{1}}\ar[l]\ar[d]\ar@(ur,dr)[]^{\mathcal{P}_{1}}\\
{\bf Rel_{0}} & {\bf Rel_{1}\ar[l]}
}
\]
in which the top horizontal arrow is a morphism of monads, and the
bottom part is the Kleisli construction of the top part. In fact,
as a consequence of the obstruction mentioned above, this diagram
expands to
\[
\xymatrix{ & {\bf Cat}\ar[rd]\ar[dl]\ar@{-->}[dd]\\
{\bf Cat_{0}}\ar[dd]\ar@<4pt>[ur]\ar@<-4pt>[ur]\ar@<4pt>[rr]\ar@<-4pt>[rr]\ar@(ul,dl)[]_{\mathcal{P}_{0}} &  & {\bf Cat_{1}}\ar[dd]\ar[ll]\ar@(ur,dr)[]^{\mathcal{P}_{1}}\\
 & {\bf Cat_{pmv}}\cong{\bf Cat_{rel}}\ar[dr]\ar[dl]\ar@{-->}[dd]\\
{\bf Rel_{0}}\ar@<4pt>[rr]\ar@<-4pt>[rr]\ar@<4pt>[ur]\ar@<-4pt>[ur]\ar[dd] &  & {\bf Rel_{1}}\ar[ll]\ar[dd]\\
 & {\bf Cat_{CSLat}}\ar[dr]\ar[dl]\\
{\bf CSLat_{0}} &  & {\bf CSLat_{1}}\ar[ll]
}
\]
with ${\bf Cat_{pmv}\cong{\bf Cat_{rel}}}$ appearing as a full subcategory
of ${\bf Rel_{1}}$. Since the Eilenberg-Moore category of $\mathcal{P}_{0}$
is ${\bf CSLat_{0}}$, the category of complete semi-lattices, we
obtain at the bottom triangle extensions of it to categories.

In mathematics, the notion of partial multivalued functor appears
in homotopy theory (see \cite{Hess}, \cite{MildTameHomotopy}) where
algebraic models are defined to be multivalued functors to ${\bf Top}$
on a category with a notion of homotopy. Multivalued functors are
also discussed in \cite{Anick}, \cite{contreras2013relational},
and \cite{conrerasHeunen}. Quite generally, \cite{catOnt} discusses partial
information, non-determinacy, and multivaluedness in many contexts.
Recently, stochastic relations are applied in such areas as computation,
automata, and languages and programming (see e.g., \cite{doberkat2003,doberkat2005,doberkat2006}).
The category of stochastic relations is the Kleisli category of the
Giry monad (\cite{doberkat2007kleisli}), extending the Manes monad.
The monad we present is a similar extension of the Manes monad not
probabilistically but rather categorically. 

The categorical prerequisites for reading this work are modest. We
refer the reader to \cite{MacLane} for general information regarding
categories, monads, the Kleisli construction, and the Eilenberg-Moore
construction. 

The plan of the paper is as follows. Section~\ref{sec:Cat1} introduces
categorders and functorders, which together form the extension ${\bf Cat_{1}}$
of ${\bf Cat}$ for the required monad to act on. Section~\ref{sec:The-power-categorder}
and Section~\ref{sec:The-direct-image} describe, respectively, the
power categorder of a categorder and the direct image functorder associated
to a functorder, which together form the underlying functor $\mathcal{P}_{1}$
of the monad on ${\bf Cat_{1}}$. Section~\ref{sec:The-monad} then
establishes the monad structure on $\mathcal{P}_{1}$, and Section~\ref{sec:partial-multivalued-functors}
presents the main result, \thmref{MainResult}, obtaining the formalism
of partial multivalued functors, as well as a brief investigation
of the Eilenberg-Moore algebras of $\mathcal{P}_{1}$ as an extension
of complete join semi lattices.

\section{Categorders and functorders\label{sec:Cat1}}

In this section we extend the category ${\bf Cat}$ of small categories
and functors to the category ${\bf Cat_{1}}$ which will serve as
the ambient category for the power object monad we seek. 
\begin{defn}
\label{def:An-ordered-category}A\emph{ category with order,} or \emph{categorder},\emph{
}is a category $\mathscr{C}$ together with an ordering of each hom
set $\mathscr{C}(C,C')$ such that for all morphisms $\xymatrix{C\ar[r]^{c_{1},c_{2}} & C'\ar[r]^{c_{1}',c_{2}'} & C''}
$, if $c_{1}\le c_{2}$ and $c_{1}'\le c_{2}'$, then $c_{1}'\circ c_{1}\le c_{2}'\circ c_{2}$.
A \emph{prefunctor }$F\colon\mathscr{C}\to\mathscr{D}$ between two
categorders consists of an assignment $F\colon{\rm ob}\mathscr{C}\to{\rm ob}\mathscr{D}$
and, for all $C,C'\in{\rm ob}\mathscr{C}$, a function $F\colon\mathscr{C}(C,C')\to\mathscr{D}(FC,FC')$.
The prefunctor $F\colon\mathscr{C}\to\mathscr{D}$ is said to be \emph{monotone
}if $Fc\le Fc'$ whenever $c\le c'$, for all $C\xrightarrow{c,c'}C'$.
Finally, a monotone prefunctor $F\colon\mathscr{C}\to\mathscr{D}$
is \emph{subfunctorial }if the conditions 
\begin{itemize}
\item $F({\rm id}_{C})\le{\rm id}_{FC}$, for all $C\in{\rm ob}\mathscr{C}$
\item $F(c'\circ c)\le Fc'\circ Fc$ for all morphisms $\xymatrix{C\ar[r]^{c} & C'\ar[r]^{c'} & C''}
$ in $\mathscr{C}$
\end{itemize}
hold. A subfunctorial monotone prefunctor will be called a \emph{functorder}. 
\end{defn}
With the obvious notions of identity functorders and composition of
functorders one obtains the category ${\bf Cat_{1}}$ of all small
categorders (where small means the underlying category is small) and
functorders. 
\begin{rem}
Notice that the notion of categorder is related to the notion of category
enriched in the category ${\bf Ord}$ of ordered sets, but the notion
of enriched functor is stronger than that of functorders, in particular
an enriched functor is a functor between the underlying categories,
while a functorder need not be. 
\end{rem}
We now describe the relationships depicted in 

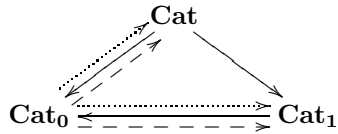
\[
\xymatrix{ & {\bf Cat}\ar[rd]\ar[dl]\\
{\bf Cat_{0}}\ar@{..>}@<4pt>[ur]\ar@{-->}@<-4pt>[ur]\ar@{..>}@<4pt>[rr]\ar@{-->}@<-4pt>[rr] &  & {\bf Cat_{1}}\ar[ll]
}
\]
taken from the top level of the diagram above. The functors on the
top left are adjoints, with left adjoint on top, but at the base of
the diagram only the dashed arrow is a right adjoint. Each of the
two triangles with either two dashed sides or two dotted sides commutes,
as does the triangle of solid arrows. In more detail, the two functors
leading from right to left are the forgetful functors which forget
the morphisms and only remember the set of objects. The functor ${\bf Cat}\to{\bf Cat_{1}}$
maps a category $\mathscr{C}$ to the categorder having $\mathscr{C}$
as underlying category, endowed with the reflexive ordering on each
hom set. Noting that every functor between categories is automatically
a functorder between the associated categorders we may identify ${\bf Cat}$
as a full subcategory of ${\bf Cat_{1}}$. More generally, while a
functorder need not have an underlying functor, any functorder $F\colon\mathscr{C}\to\mathscr{D}$
to a category $\mathscr{D}$ is a functor between the underlying categories.
The adjoints of the forgetful functors are given by the discrete and
indiscrete categories associated to a set. Again, the obvious discrete
categorder construction ${\bf Cat_{0}}\to{\bf Cat_{1}}$ constitutes
a functor, but it is no longer left adjoint to the forgetful functor.

\section{The power categorder of a categorder\label{sec:The-power-categorder}}

The suitable extension of the power set construction to categories
is the aim of this section. Of course, the concept we develop is the
one achieving the goal of this work. The possibility of defining the
power category of a category to consist of all subcategories of the
given category is explored in \cite{powerCat}. 

Firstly, we extend the power set construction to operate on ordered
sets instead of just sets. A\emph{ }subset $A\subseteq S$ is \emph{down
closed} if $y\in A$ and $x\le y$ imply $x\in A$, for all $x,y\in S$.
Then $\mathcal{P}_{0}S$, for an ordered set $S$ is the collection
of all down closed subsets $A\subseteq S$. $\mathcal{P}_{0}S$ is
endowed with the ordering given by set inclusion. We note that $s\mapsto{\downarrow}s=\{t\in S\mid t\le s\}$
is an order embedding $S\to\mathcal{P}_{0}S$. In the special case
where the ordering on $S$ is the reflexive relation, we note that
$\mathcal{P}_{0}(S)$ is just the power set of $S$ and the embedding
$S\mapsto\mathcal{P}_{0}S$ is the inclusion by singletons $s\mapsto\{s\}$.
More generally, for any subset $A\subseteq S$, let ${\downarrow}A=\bigcup_{a\in A}({\downarrow}a)$,
called the \emph{down closure }of $A$. 

Let us now fix a small categorder $\mathscr{C}$, for which we define
a categorder $\mathcal{P}_{1}\mathscr{C}$. Let ${\rm ob}(\mathcal{P}_{1}\mathscr{C})=\mathcal{P}_{0}({\rm ob}(\mathscr{C}))$,
namely all collections $\mathcal{C}$ of objects of $\mathscr{C}$.
Since hom sets are pair-wise disjoint, the orderings on each $\mathscr{C}(C,C')$
induce an ordering on ${\rm mor}\mathscr{C}$. The following down
closures are computed in this ambient ordered set. For objects $\mathcal{C}$
and $\mathcal{C}'$ in ${\rm ob}(\mathcal{P}_{1}\mathscr{C})$ let
${\rm id}_{\mathcal{C}}={\downarrow}\{{\rm id}_{C}\mid C\in\mathcal{C}\}$
and let $\mathscr{C}(\mathcal{C},\mathcal{C}')=\bigcup_{C\in\mathcal{C},C'\in\mathcal{C}'}\mathscr{C}(C,C')$.
The morphisms in $\mathcal{P}_{1}\mathscr{C}$ are then given by $(\mathcal{P}_{1}\mathscr{C})(\mathcal{C},\mathcal{C}')=\mathcal{P}_{0}(\mathscr{C}(\mathcal{C},\mathcal{C}'))$,
where we use the $\mathcal{P}_{0}$ construction on ordered sets.
In other words, a morphism in $(\mathcal{P}_{1}\mathscr{C})(\mathcal{C},\mathcal{C}')$
is a down closed collection ${\bf c}$ of morphisms in $\mathscr{C}$
with appropriate domains and codomains. The composition of $\xymatrix{\mathcal{C}\ar[r]^{{\bf c}} & \mathcal{C}'\ar[r]^{{\bf c'}} & \mathcal{C}''}
$ is ${\bf c}'\circ{\bf c}={\downarrow}\{c'\circ c\mid c\in{\bf c},c'\in{\bf c'}\}$,
i.e., the down closure of the set of all morphisms obtained as a composition
of morphisms $c\in{\bf c}$ and $c'\in{\bf c'}$ (whenever the composition
exists). 
\begin{rem}
\label{rem:MoreMorphismsThanP_0ofMorC}It is important to note that
${\rm mor}\mathcal{P}_{1}\mathscr{C}$ is not just $\mathcal{P}_{0}({\rm mor}\mathscr{C})$.
For instance, the empty set of morphisms is a morphism in $(\mathcal{P}_{1}\mathscr{C})(\mathcal{C},\mathcal{C}')$
for all objects $\mathcal{C},\mathcal{C}'\in{\rm ob}(\mathcal{P}_{1}\mathscr{C})$.
Many other sets of morphisms appear as a single morphism in different
hom sets in $\mathcal{P}_{1}\mathscr{C}$, so, of course, we silently
mark them differently so as to obtain a category. Consequently, ${\rm mor}(\mathcal{P}_{1}\mathscr{C})$
is significantly larger than $\mathcal{P}_{0}({\rm mor}\mathscr{C})$.
\end{rem}
The following result justifies calling $\mathcal{P}_{1}\mathscr{C}$
the \emph{power categorder} of $\mathscr{C}$. 
\begin{thm}
$\mathcal{P}_{1}\mathscr{C}$ is a categorder. \end{thm}
\begin{proof}
The verification is straightforward. We only address the arguments
that make use of the monotonicity in each variable of the composition
in a categorder (a property that follows at once from the definition),
starting with the associativity of composition in $\mathcal{P}_{1}\mathscr{C}$.
Consider $\mathcal{C}\xrightarrow{{\bf c}}\mathcal{C}'\xrightarrow{{\bf c'}}\mathcal{C}''\xrightarrow{{\bf c''}}\mathcal{C}'''$,
for which we show that $({\bf c''}\circ{\bf c'})\circ{\bf c}={\bf c''}\circ{\bf c'}\circ{\bf c}={\bf c''}\circ({\bf c'}\circ{\bf c})$,
where we introduce the ternary composition ${\bf c''}\circ{\bf c'}\circ{\bf c}={\downarrow}\{c''\circ c'\circ c\mid c\in{\bf c},c'\in{\bf c'},c''\in{\bf c''}\}$.
Let $\psi\in({\bf c''}\circ{\bf c'})\circ{\bf c}$, i.e., $\psi\le\varphi\circ c$,
where $\varphi\in{\bf c''}\circ{\bf c'}$ and $c\in{\bf c}$. Thus,
$\varphi\le c''\circ c'$, for some $c'\in{\bf c'}$ and $c''\in{\bf c''}$.
It now follows that $\psi\le\varphi\circ c\le(c''\circ c')\circ c=c''\circ c'\circ c$,
and thus $\psi\in{\bf c''}\circ{\bf c'}\circ{\bf c}$. We thus conclude
that $({\bf c''}\circ{\bf c'})\circ{\bf c}\subseteq{\bf c''}\circ{\bf c'}\circ{\bf c}$.
The reverse inclusion is immediate, so equality follows, and the second
claimed equality is obtained similarly. The neutrality of the identities
${\rm id}_{\mathcal{C}}$ is a similar argument. 
\end{proof}
Notice that ${\rm id}_{\{C\}}={\downarrow}\{{\rm id}_{C}\}$ and that
$\mathscr{C}(\{C\},\{C'\})=\mathscr{C}(C,C')$, and thus $(\mathcal{P}_{1}\mathscr{C})(\{C\},\{C'\})=\mathcal{P}_{0}(\mathscr{C}(C,C'))$
(again, we mean the order theoretic version of $\mathcal{P}_{0}$),
for all $C,C'\in{\rm ob}(\mathscr{C})$. Consequently, the assignment
$C\mapsto\{C\}$ and $C\xrightarrow{c}C'\mapsto\{C\}\xrightarrow{{\downarrow}\{c\}}\{C'\}$
is defined and constitutes a functorder $\eta_{\mathscr{C}}\colon\mathscr{C}\to\mathcal{P}_{1}(\mathscr{C})$,
essentially the inclusion by singletons. Of particular importance
to the proof of \thmref{MainResult} below, we note that when $\mathscr{C}$
is a category, i.e., a categorder with the reflexive order on each
hom set, we have that $(\mathcal{P}_{1}\mathscr{C})(\mathcal{C},\mathcal{C}')=\mathcal{P}_{0}(\mathscr{C}(\mathcal{C},\mathcal{C}'))$,
the ordinary power set ordered by inclusion. 

We note that $\mathcal{P}_{1}$ generally does not preserve discrete
categories. In fact, the only discrete category $\mathscr{D}$ such
that $\mathcal{P}_{1}\mathscr{D}$ is itself discrete is $\mathscr{D}=\emptyset$,
the initial category. The precise way in which $\mathcal{P}_{1}$
extends $\mathcal{P}_{0}$ along the discrete category construction
$\Delta\colon{\bf Cat_{0}}\to{\bf Cat_{1}}$ is detailed in the following
result. 
\begin{lem}
\label{lem:P1ExtP0Obj}For every set $S$ there exists a faithful functorder $\iota_{S}\colon(\Delta\circ\mathcal{P}_{0})S\to(\mathcal{P}_{1}\circ\Delta)S$
which is the identity on objects. In particular, ${\rm ob}(\mathcal{P}_{1}\Delta S)=\mathcal{P}_{0}S$. \end{lem}
\begin{proof}
The fact that ${\rm ob}(\mathcal{P}_{1}\Delta S)=\mathcal{P}_{0}S$
is trivial. Extending the identity function to a functorder $\Delta\mathcal{P}_{0}S\to\mathcal{P}_{1}\Delta S$
is trivially obtained by inclusion of singletons. 
\end{proof}

\section{The direct image functorder\label{sec:The-direct-image} }

With the power categorder constructed we now address the functoriality
of $\mathcal{P}_{1}$.  We first construct, for a given functorder
$F\colon\mathscr{C}\to\mathscr{D}$, its associated direct image functorder
$F_{\shortrightarrow}\colon\mathcal{P}_{1}\mathscr{C}\to\mathcal{P}_{1}\mathscr{D}$,
which is given by $F_{\shortrightarrow}(\mathcal{C})=\{FC\mid C\in\mathcal{C}\}$,
for all $\mathcal{C}\in{\rm ob}(\mathcal{P}_{1}\mathscr{C})$, and
$F_{\shortrightarrow}({\bf c})={\downarrow}\{Fc\mid c\in{\bf c}\}$,
for all morphisms ${\bf c}$ of $\mathcal{P}_{1}\mathscr{C}$. 
\begin{thm}
$F_{\shortrightarrow}\colon\mathcal{P}_{1}\mathscr{C}\to\mathcal{P}_{1}\mathscr{D}$
is a functorder.\end{thm}
\begin{proof}
The verification is straightforward. Let us for instance verify that
$F_{\shortrightarrow}({\rm id}_{\mathcal{C}})\le{\rm id}_{F_{\shortrightarrow}\mathcal{C}}$,
where $\mathcal{C}\in{\rm ob}\mathcal{P}_{1}\mathscr{C}$, remembering
that $\le$ is given by set inclusion. Indeed, $F_{\shortrightarrow}({\rm id}_{\mathcal{C}})={\downarrow}(F({\downarrow}(\{{\rm id}_{C}\mid C\in\mathcal{C}\})))$
while ${\rm id}_{F_{\shortrightarrow}\mathcal{C}}={\downarrow}\{{\rm id}_{FC}\mid C\in\mathcal{C}\}$.
Thus, if $d\in F_{\shortrightarrow}({\rm id}_{\mathcal{C}})$, then
$d\le Fc$, where $c\in{\rm mor}\mathscr{C}$ with $c\le{\rm id}_{C}$
for some $C\in\mathcal{C}$. It then follows that $d\le Fc\le F({\rm id}_{C})\le{\rm id}_{FC}$,
and thus $d\in{\rm id}_{F_{\shortrightarrow}\mathcal{C}}$, as needed. \end{proof}
\begin{rem}
\label{rem:WhyFunctorders}Note that $F_{\shortrightarrow}$ need
not be a functor, even if $F$ is. This can be seen by computing $F_{\shortrightarrow}$
for the functor $K$ described in Section~\ref{sec:Introduction}. \end{rem}
\begin{thm}
Defining $\mathcal{P}_{1}(F)=F_{\shortrightarrow}$ turns $\mathcal{P}_{1}$
into a covariant endofunctor on ${\bf Cat_{1}}$.\end{thm}
\begin{proof}
Again, a straightforward verification shows that indeed $({\rm id}_{\mathscr{C}})_{\shortrightarrow}={\rm id}_{\mathscr{C}}$
and $(G\circ F)_{\shortrightarrow}=G_{\shortrightarrow}\circ F_{\shortrightarrow}$
hold true for all categorders $\mathscr{C}$ and functorders $\mathscr{C}\xrightarrow{G}\mathscr{D}\xrightarrow{F}\mathscr{E}$. 
\end{proof}
\remref{WhyFunctorders} can be restated now as the fact that $\mathcal{P}_{1}\colon{\bf Cat_{1}}\to{\bf Cat_{1}}$
does not restrict to an endofunctor on the subcategory ${\bf Cat}$
of categories and functors. This is precisely the reason why ${\bf Cat}$
was extended to ${\bf Cat_{1}}$. 

The relationship between the functors $\mathcal{P}_{0}$ and $\mathcal{P}_{1}$
along the discrete categorder functor $\Delta\colon{\bf Cat_{0}}\to{\bf Cat_{1}}$
and the forgetful functor $G\colon{\bf Cat_{1}}\to{\bf Cat_{0}}$
is summarized next. The trivial proof is omitted
\begin{lem}
\label{lem:P1ExtP2Mor}The functorders constructed in \lemref{P1ExtP0Obj}
form the components of a natural transformation $\iota\colon\Delta\circ\mathcal{P}_{0}\to\mathcal{P}_{1}\circ\Delta$.
Further, $\mathcal{P}_{0}=G\circ\mathcal{P}_{1}\circ\Delta$. 
\end{lem}

\section{The power categorder monad\label{sec:The-monad}}

In this section we present the monad structure on the power categorder
functor $\mathcal{P}_{1}$. Recall that the Manes monad structure on $\mathcal{P}_{0}\colon{\bf Cat_{0}}\to{\bf Cat_{0}}$
is given by the unit $\eta\colon{\rm id}_{{\bf Cat_{0}}}\to\mathcal{P}_{0}$
with components $\eta_{S}\colon S\to\mathcal{P}_{0}S$ with
$s\mapsto\{s\}$, and multiplication $\mu\colon\mathcal{P}_{0}^{2}\to\mathcal{P}_{0}$
given by the components $\mu_{S}\colon\mathcal{P}_{0}^{2}S\to\mathcal{P}_{0}S$
by taking unions. It is not hard to see that our extension above of
$\mathcal{P}_{0}$ to a functor ${\bf Ord}\to{\bf Ord}$ on the category
of ordered sets and order preserving functions also gives rise to
a monad, where the unit $\eta\colon{\rm id}_{{\bf Ord}}\to\mathcal{P}_{0}$
is given by $\eta_{S}(s)={\downarrow} \{s\}$, and $\mu_{S}$ is still
given by unions. 

In Section~\ref{sec:The-power-categorder} we already obtained the
components $\eta_{\mathscr{C}}\colon\mathscr{C}\to\mathcal{P}_{1}\mathscr{C}$,
given by $C\mapsto\{C\}$ and $c\mapsto{\downarrow}\{c\}$, and it
is easy to verify that they form a natural transformation ${\rm id}_{{\bf Cat_{1}}}\to\mathcal{P}_{1}$.
We shall now obtain a multiplication natural transformation $\mu\colon\mathcal{P}_{1}^{2}\to\mathcal{P}_{1}$.
Recall that we extended the notation $\mathscr{C}(C,C')$ to collection
$\mathcal{C},\mathcal{C}'$ of objects by means of $\mathscr{C}(\mathcal{C},\mathcal{C}')=\bigcup_{C\in\mathcal{C},C'\in\mathcal{C}'}\mathscr{C}(C,C')$.
We may further extend the notation to families $\mathbb{C},\mathbb{C}'$
of collections of objects by defining $\mathscr{C}(\mathbb{C},\mathbb{C}')=\bigcup_{\mathcal{C}\in\mathbb{C},\mathcal{C}'\in\mathbb{C}'}\mathscr{C}(\mathcal{C},\mathcal{C}')$. 

Looking into the structure of $\mathcal{P}_{1}^{2}\mathscr{C}$ we
have ${\rm ob}(\mathcal{P}_{1}^{2}\mathscr{C})=\mathcal{P}_{0}^{2}({\rm ob}\mathscr{C})$
and $(\mathcal{P}_{1}^{2}\mathscr{C})(\mathbb{C},\mathbb{C}')=\mathcal{P}_{0}((\mathcal{P}_{1}\mathscr{C})(\mathbb{C},\mathbb{C}'))=\mathcal{P}_{0}^{2}(\mathscr{C}(\mathbb{C},\mathbb{C}'))$.
Thus, given a categorder $\mathscr{C}$, we define the following functorder
$\mu_{\mathscr{C}}\colon\mathcal{P}_{1}^{2}\mathscr{C}\to\mathcal{P}_{1}\mathscr{C}$.
On objects $\mu_{\mathscr{C}}\colon\mathcal{P}_{0}(\mathcal{P}_{0}({\rm ob}\mathscr{C}))\to\mathcal{P}_{0}({\rm ob}\mathscr{C})$
is given by $\mu_{{\rm ob}(\mathscr{C})}$, the component of the Manes multiplication
for $\mathcal{P}_{0}$. On morphisms, $\mu_{\mathscr{C}}\colon(\mathcal{P}_{1}^{2}\mathscr{C})(\mathbb{C},\mathbb{C}')\to(\mathcal{P}_{1}\mathscr{C})(\mu_{\mathscr{C}}\mathbb{C},\mu_{\mathscr{C}}\mathbb{C}')$
is the function $\mathcal{P}_{0}^{2}(\mathscr{C}(\mathbb{C},\mathbb{C}'))\to\mathcal{P}_{0}(\mathscr{C}(\mu_{\mathscr{C}}\mathbb{C},\mu_{\mathscr{C}}\mathbb{C}'))$
given by $\mu_{\mathscr{C}(\mathbb{C},\mathbb{C}')}$, the component
of the multiplication for $\mathcal{P}_{0}$ as a monad on $\bf Ord$, noting indeed that $\mathscr{C}(\mathbb{C},\mathbb{C}')=\mathscr{C}(\mu_{\mathscr{C}}\mathbb{C},\mu_{\mathscr{C}}\mathbb{C}')$. 
\begin{thm}
The functor $\mathcal{P}_{1}\colon{\bf Cat_{1}}\to{\bf Cat_{1}}$
together with the natural transformations $\eta\colon{\rm id}_{{\bf Cat_{1}}}\to\mathcal{P}_{1}$
and $\mu\colon\mathcal{P}_{1}^{2}\to\mathcal{P}_{1}$ form a monad. \end{thm}
\begin{proof}
This is once more a rather straightforward task, very similar to the
proof that the Manes monad $\mathcal{P}_{0}\colon{\bf Set}\to{\bf Set}$
is a monad. 
\end{proof}
The power categorder monad $\mathcal{P}_{1}$ is related to the power
set monad $\mathcal{P}_{0}$ in the following manner, which is
the formal presentation of the claim that the monad $\mathcal{P}_{0}\colon{\bf Set}\to{\bf Set}$
is the object part of $\mathcal{P}_{1}$. Let us recall the notion
of a functor of monads (see \cite{FormalMonads}). Given a monad $(S,\mu,\eta)$
on a category ${\bf C}$ and a monad $(T,\mu,\eta)$ on a category
${\bf D}$, a functor of monads is a functor $F\colon{\bf D}\to{\bf C}$
together with a natural transformation $\sigma\colon S\circ F\to F\circ T$
satisfying $F\eta=\sigma\circ\eta F$ and $\sigma\circ\mu F=F\mu\circ\sigma T\circ S\sigma$.
It is trivial to verify that the forgetful functor $G\colon{\bf Cat_{1}}\to{\bf Cat_{0}}$
is a monad functor when considered with the natural transformation
$\sigma\colon\mathcal{P}_{0}\circ G\to G\circ\mathcal{P}_{1}$, which
in fact is a natural isomorphism. In contrast, the discrete categorder
functor $\Delta\colon{\bf Cat_{0}}\to{\bf Cat_{1}}$ together with
the natural transformation $\iota$ (see \lemref{P1ExtP0Obj} and
\lemref{P1ExtP2Mor}) fail to form a monad functor, since $\iota$
relates the natural transformations in the wrong order. In other words,
the monad $\mathcal{P}_{0}\colon{\bf Set}\to{\bf Set}$ does not embed
in $\mathcal{P}_{1}$.

\section{partial multivalued functors\label{sec:partial-multivalued-functors}}

In this section we obtain ${\bf Cat_{1}}$ as a framework for partial
multivalued functors, analogously to how ${\bf Set}$ serves as a
framework for partial multivalued functions. 

Let ${\bf Rel_{1}}$ be the Kleisli category of the monad $\mathcal{P}_{1}\colon{\bf Cat_{1}\to{\bf Cat_{1}}}$.
Thus, the objects of ${\bf Rel_{1}}$ are all small categorders, and
morphisms $F\colon\mathscr{C}\tobar\mathscr{D}$ are given by morphisms
$\hat{F}\colon\mathscr{C}\to\mathcal{P}_{1}\mathscr{D}$ in ${\bf Cat_{1}}$.
The identity ${\rm id}_{\mathscr{C}}\colon\mathscr{C}\tobar\mathscr{C}$
at $\mathscr{C}$ is the component $\eta_{\mathscr{C}}\colon\mathscr{C}\to\mathcal{P}_{1}\mathscr{C}$
of the natural transformation from Section~\ref{sec:The-monad},
and the composition of $F\colon\mathscr{C}\tobar\mathscr{D}$ with
$G\colon\mathscr{D}\tobar\mathscr{E}$ is given by $\mathscr{C}\xrightarrow{\hat{F}}\mathcal{P}_{1}\mathscr{D}\xrightarrow{\mathcal{P}_{1}G}\mathcal{P}_{1}\mathcal{P}_{1}\mathscr{E}\xrightarrow{\mu_{\mathscr{E}}}\mathcal{P}_{1}\mathscr{E}$,
utilizing the multiplication natural transformation $\mu$ defined
in Section~\ref{sec:The-monad}. 
The following result is the main theorem of this work where the category of categories and partial multivalued functors is identified as a full subcategory of $\bf Rel_1$.
\begin{thm}
\label{thm:MainResult}The category ${\bf Cat_{pmv}}$ is isomorphic
to the full subcategory of ${\bf Rel_{1}}$ spanned by the categories. 
\end{thm}
\begin{proof}
The claimed isomorphism is the identity on objects, and maps a morphism
$F\colon\mathscr{C}\to\mathscr{D}$ between two categories in ${\bf Rel_{1}}$,
i.e., a functorder $\hat{F}\colon\mathscr{C}\to\mathcal{P}_{1}\mathscr{D}$,
to the following partial multivalued functor $G\colon\mathscr{C}\tobar\mathscr{D}$.
Let us write $F_{0}$ and $F_{1}$ for the object and morphism parts,
respectively, of the functorder $\hat{F}$. Thus, $F_{0}\colon{\rm ob}\mathscr{C}\to\mathcal{P}_{0}({\rm ob}\mathscr{D})$
is a function, giving rise to a partial multivalued function $G_{0}\colon{\rm ob}\mathscr{C}\tobar{\rm ob}\mathscr{D}$.
The morphism part of $\hat{F}$ consists of components $F_{1}\colon\mathscr{C}(C,C')\to\mathcal{P}_{0}(\mathscr{D}(F_{0}C,F_{0}C'))$,
giving rise to a partial multivalued function $G_{1}\colon\mathscr{C}(C,C')\tobar\mathscr{D}(G_{0}C,G_{0}C')$.
The condition that if $D\xrightarrow{d}D'\in G_{1}(C\xrightarrow{c}C')$,
then $D\in G_{0}C$ and $D'\in G_{0}C'$ follows at once. Further,
since $\hat{F}({\rm id}_{C})\le{\rm id}_{FC}$ translates to $\hat{F}({\rm id}_{C})\subseteq\{{\rm id}_{D}\mid D\in FC\}$,
it follows that if $D\xrightarrow{d}D'\in G_{1}({\rm id}_{C})$, then
$D=D'$ and $d={\rm id}_{D}$. Finally, from $\hat{F}(c'\circ c)\le\hat{F}c'\circ\hat{F}c$,
i.e., $\hat{F}(c'\circ c)\subseteq\hat{F}c'\circ\hat{F}c$, it follows
that if $d''\in G_{1}(c'')$, and $c''=c'\circ c$, i.e., $d\in\hat{F}(c'\circ c)$,
then $d\in\hat{F}c'\circ\hat{F}c$, implying the existence of $d\in G_{1}c$
and $d'\in G_{1}c'$ with $d''=d'\circ d$. The verification that
the composition in ${\bf Rel_{1}}$ agrees with the composition in
${\bf Cat_{pmv}}$ and that identities agree, is easy. Finally, constructing
a morphism in ${\bf Rel_{1}}$ from a given partial multivalued functor
is similar. 
\end{proof}
Referring to the middle layer of the main diagram from \secref{Introduction},
we first mention that the same commutativity relations hold as for
the top layer (and these were described in \secref{Cat1}). The functors
themselves are the obvious ones; the functors going from right to
left are the forgetful functors (forgetting morphisms) and the functors
emanating from ${\bf Rel_{0}}$ are the immediate versions of the
discrete and indiscrete constructions. However, due to a phenomenon
related to the failure of $\mathcal{P}_{1}$ to preserve discreteness,
these functors are no longer left, respectively right, adjoint to
the forgetful functors (as is easily verified). We note (cf. the closing
discussion in \secref{The-monad}) that ${\bf Rel_{1}}\to{\bf Rel_{0}}$
is the functor induced by the fact that $\mathcal{P}_{0}$ is the
object part of $\mathcal{P}_{1}$, while neither of the functors in
the other direction is induced by an immediate structural relationship
between the two monads. Finally, the functors leading from the top
layer to the middle layer all view a single-valued entity (i.e., function,
functor, or functorder) as a particular kind of a partial multivalued
entity in the corresponding category in the middle layer.

Recall (\cite{Manes}) that the Eilenberg-Moore category of $\mathcal{P}_{0}$
is the category ${\bf CSLat_{0}}$ of complete join semi lattices
and join preserving homomorphisms. Considering the Eilenberg-Moore
category of $\mathcal{P}_{1}$ leads to the third and final layer
of the diagram. To be precise, ${\bf CSLat_{1}}$ is the Eilenberg-Moore
category of $\mathcal{P}_{1}$, and ${\bf Cat_{CSLat}}$ is its full
subcategory spanned by the categories. An object in ${\bf CSLat_{1}}$
is thus a categorder $\mathscr{C}$ together with a structure functorder
$a\colon\mathcal{P}_{1}\mathscr{C}\to\mathscr{C}$, satisfying a unit
and associativity rules. In some more detail, the structure functorder
$a$ consists of a function $a\colon\mathcal{P}_{0}({\rm ob}\mathscr{C})\to{\rm ob}\mathscr{C}$
and, for all $\mathcal{C},\mathcal{C}\subseteq{\rm ob}\mathscr{C}$,
a function $\mathcal{P}_{0}(\mathscr{C}(\mathcal{C},\mathcal{C}'))\to\mathscr{C}(a(\mathcal{C}),a(\mathcal{C}'))$.
Note that the objects of ${\bf Cat_{CSLat}}$ are defined purely categorically,
rather than categorderly, since a functorder $\mathcal{P}_{1}\mathscr{C}\to\mathscr{C}$
for a category $\mathscr{C}$ is just a functor of the underlying
categories. 

The relationship between the monads $\mathcal{P}_{0}$ and $\mathcal{P}_{1}$
implies that every $\mathcal{P}_{1}$-algebra gives rise to a $\mathcal{P}_{0}$-algebra,
namely forgetting the morphisms of an object $\mathscr{C}$ in ${\bf CSLat_{1}}$
results in a complete join semi lattice. Moreover, (and here \remref{MoreMorphismsThanP_0ofMorC}
is of importance), each hom set $\mathscr{C}(C,C')$ carries the structure
of a complete join semi lattice. ${\bf Cat_{CSLat}}$ is thus clearly
related to the category of categories internal to ${\bf CSLat}$,
though there is much more to $\mathscr{C}$ than just the existence
of complete join semi lattice structures on the objects and on the
hom sets. Fully investigating these structures will take us too far
afield though, and so, other than noting that the functors leading
from layer two to layer three are the usual embeddings of the Kleisli
category in the Eilenberg-Moore category,  we conclude by discussing
the omittance of a discrete and indiscrete constructions at the bottom
layer. For instance, given a complete join semi lattice $S$ with
three elements $u,v,w$ with $u\vee w\ne v\vee w$, the discrete category
$\Delta S$ does not carry the structure of an object in ${\bf Cat_{CSLat}}$
extending the lattice structure of $S$. Indeed, if $a\colon\mathcal{P}_{1}\Delta S\to\Delta S$
were such a structure, then noting than $\{{\rm id}_{w}\}\colon\{u,w\}\to\{v,w\}$
is a morphism in $\mathcal{P}_{1}\Delta S$, one must then have $a(\{{\rm id}_{w}\})\colon u\vee w\to v\vee w$
in $\Delta S$, but no such morphism exists. 

\bibliographystyle{plain}
\bibliography{generalReferences}

\end{document}